\documentclass[a4paper,12pt]{article}
\usepackage{amssymb, amsmath, latexsym,color} 
\usepackage{amsthm}
\usepackage{mathtools, mathrsfs}
\usepackage{color}
\usepackage{epsfig}
\usepackage{latexsym}
\usepackage{graphicx}
\usepackage{caption}
\usepackage{subcaption}
\usepackage{bm}
\usepackage{tcolorbox}
\usepackage[affil-it]{authblk}
\usepackage{IEEEtrantools}

\newtheorem{theorem}{Theorem}
\newtheorem{lemma}[theorem]{Lemma}

\newtheorem{corollary}[theorem]{Corollary}

\DeclarePairedDelimiter{\ceil}{\lceil}{\rceil}
\DeclarePairedDelimiter{\floor}{\lfloor}{\rfloor}

\DeclareMathOperator{\sds}{sds}
\DeclareMathOperator{\maxsds}{maxsds}

\renewcommand{\geq}{\geqslant}
\renewcommand{\leq}{\leqslant}
\renewcommand{\ge}{\geqslant}
\renewcommand{\le}{\leqslant}
\def\eps{\varepsilon}

\begin{document}

\author{Carly Bodkin \thanks{\texttt{carly.bodkin@monash.edu}}
\thanks{Supported by an Australian Government Research Training Program (RTP) Scholarship.}
}
\affil{\small School of Mathematics\\Monash University\\Clayton Vic 3800 Australia}
\author{Anita Liebenau \thanks{\texttt{a.liebenau@unsw.edu.au}}
\thanks{Supported by the Australian Research Council grant 
DE170100789 and DP180103684.}
}
\affil{\small School of Mathematics and Statistics, UNSW Sydney NSW 2052 Australia} 
\author{Ian M. Wanless \thanks{\texttt{ian.wanless@monash.edu}}
\thanks{Supported by the Australian Research Council grant DP150100506.}
}
\affil{\small School of Mathematics\\Monash University\\Clayton Vic 3800 Australia}

\date{}
\title{Most binary matrices have no small defining set}
\maketitle

\begin{abstract}
Consider a matrix $M$ chosen uniformly at random from a class of
$m\times n$ matrices of zeros and ones with prescribed row and column
sums.  A partially filled matrix $D$ is a \emph{defining set} for $M$
if $M$ is the unique member of its class that contains the entries in
$D$.  The {\em size} of a defining set is the number of filled
entries. A {\em critical set} is a defining set for which the removal of
any entry stops it being a defining set.

For some small fixed $\eps>0$, we assume that $n\le m=o(n^{1+\eps})$,
and that $\lambda\le1/2$, where $\lambda$ is the proportion of entries
of $M$ that equal $1$. We also assume that the row sums of $M$ do not
vary by more than $\mathcal{O}(n^{1/2+\eps})$, and that the column
sums do not vary by more than $\mathcal{O}(m^{1/2+\eps})$.  Under
these assumptions we show that $M$ almost surely has no defining set
of size less than $\lambda mn-\mathcal{O}(m^{7/4+\eps})$. It follows
that $M$ almost surely has no critical set of size more than
$(1-\lambda)mn+\mathcal{O}(m^{7/4+\eps})$. Our results generalise
a theorem of Cavenagh and Ramadurai, who examined the case when
$\lambda=1/2$ and $n=m=2^k$ for an integer $k$.
\end{abstract}

\section{Introduction}

Let $m$ and $n$ be integers, and let ${\bf s}=(s_1,s_2,\dots,s_{m})$ and ${\bf t}=(t_1,t_2,\dots,t_{n})$ be vectors of non-negative integers. Then $\mathcal{A}({\bf s},{\bf t})$ is defined to be the set of all $m \times n$ binary matrices with $s_i$ ones in row $i$ and $t_j$ ones  in column $j$, where $1 \leq i \leq m$ and $1 \leq j \leq n$.
We say \emph{almost all} matrices in
$\mathcal{A}({\bf s},{\bf t})$ have a property if the probability that
a matrix chosen uniformly at random from $\mathcal{A}({\bf s},{\bf t})$
has the property tends to $1$ as $m,n \rightarrow \infty$.

A \emph{partial} binary matrix is a matrix $M$ with entries $0$, $1$ or $\star$, where we call a cell empty if its
entry is $\star$. Let $\mathcal{A}'({\bf s},{\bf t})$ denote the set of all $m \times n$ partial binary matrices with at most $s_i$ ones and $n-s_i$ zeros in row $i$, and at most $t_j$ ones and $m-t_j$ zeros in column $j$. Given $M \in \mathcal{A}({\bf s},{\bf t})$ and $D=[D_{ij}] \in \mathcal{A}'({\bf s},{\bf t})$ we write $D \subseteq M$ if $D_{ij} \in \{M_{ij}, \star\}$, for all $1 \leq i \leq m$ and $1 \leq j \leq n$.

Suppose $D \in \mathcal{A}'({\bf s},{\bf t})$ and $M \in \mathcal{A}({\bf s},{\bf t})$. Then we say $D$ is a \emph{defining set} for $M$ if $M$ is the unique member of $\mathcal{A}({\bf s},{\bf t})$ such that $D \subseteq M$. Furthermore, for $D \subseteq M$ we define the partial matrix $M \setminus D \in \mathcal{A}'({\bf s},{\bf t})$ by
$$
(M \setminus D)_{ij} =
\begin{cases}
M_{ij} & \text{if  } D_{ij} = \star \\
\star & \text{otherwise}.
\end{cases}$$
The \emph{size} of a partial binary matrix $D$, denoted $|D|$, is the number of nonempty cells. We define 
\begin{align*}
\sds(M) &= \min\{|D| : D\text{ is a defining set for }M\}, \\
\maxsds({\bf s},{\bf t}) &= \max \{\sds(M) : M \in \mathcal{A}({\bf s},{\bf t})\}.
\end{align*}
Also, define $\maxsds(m,n)$ to be the maximum of $\sds$ amongst all $m\times n$ binary matrices.

For integers $k$ and $n$, let $\Lambda^{k}_{n}$ be the set of all $n
\times n$ binary matrices with constant row and column sum $k$. In
\cite{CavRam18}, Cavenagh and Ramadurai construct a matrix in
$\Lambda^{k}_{2k}$ with no defining set of size less than
$2k^2-\mathcal{O}(k^{7/4})$ whenever $k$ is a power of $2$.  In
\S\ref{s:prfmain}, we prove a similar result for almost all
matrices $M\in\mathcal{A}({\bf s},{\bf t})$, for every pair of
integers $m$ and $n$, provided $M$ is not too far from square, and the
number of ones in each row and column does not stray too far from the
average values $s$ and $t$. Our result is this:

\begin{theorem}\label{main-result}
  Fix a sufficiently small $\eps >0$. For integers
  $m,n \rightarrow \infty$ with $n\le m=o(n^{1+\eps})$,
  let ${\bf s}=(s_1,s_2,\dots,s_{m})$ and ${\bf  t}=(t_1,t_2,\dots,t_{n})$
  be vectors of positive integers such that
  $\sum^m_{i=1}s_i = \sum^{n}_{j=1}t_j$. Define
  $s=m^{-1}\sum^{m}_{i=1}s_i$ and $t=n^{-1}\sum^{n}_{j=1}t_j$ and suppose that $|s_i-s| =\mathcal{O}(n^{1/2+\eps})$
  uniformly for $1 \leq i \leq m$, and
  $|t_j-t| = \mathcal{O}(m^{1/2+\eps})$ uniformly for $1 \leq j \leq n$. Suppose
$\lambda = s/n=t/m\leq 1/2$ and that $\lambda$ is bounded away from zero. 
  Also suppose that 
$$\dfrac{(1-2\lambda)^2}{4\lambda(1-\lambda)}\left( 1+ \dfrac{5m}{6n} + \dfrac{5n}{6m}\right) \leq  \dfrac{\log m}{3}.$$
Then almost all matrices in $\mathcal{A}({\bf s},{\bf t})$ have no defining set of size less than $\lambda mn-\mathcal{O}(m^{7/4+\eps})$.
\end{theorem}

This result significantly generalises the theorem of Cavenagh and Ramadurai mentioned above, albeit with a slightly worse error term. Taking $m=n=2k$ and $s_i=t_j=k$ for all $i,j$, Theorem \ref{main-result} implies the following corollary.

\begin{corollary}\label{Cor}
For any integer $k$,  almost all matrices $M \in \Lambda^{k}_{2k}$ have no defining set of size smaller than $2k^2 - \mathcal{O}(k^{7/4+\eps})$.
\end{corollary} 
We refer to the parameter $\lambda$ in Theorem~\ref{main-result} as the density of
$M \in \mathcal{A}({\bf s},{\bf t})$. It is the proportion of entries in $M$ which equal
one. Throughout this paper we require $\lambda$ to be bounded away from
zero. Our approach relies on an asymptotic formula from \cite{Canfield} for the number of
bipartite graphs with a given degree sequence. Similar enumeration
results do exist for the very sparse range~\cite{GMW06}, but the
intermediate range is not yet covered.  This is why we decided to not consider the
case when $\lambda\rightarrow0$.
Furthermore, we will assume that $\lambda \leq 1/2$. Without that
assumption, our problem
has symmetry between zeros and ones in the sense that we may
switch zeros and ones without changing the size of the smallest
defining set.  We can easily form a defining set for any matrix $M$ by taking
either all the ones or all the zeros in $M$.  The smaller of
these two options turns out to provide a good upper bound on the
size of the smallest defining set. The justification for legislating
that $\lambda\le1/2$ is that it simplifies the exposition if we know that
the number of ones does not exceed the number of zeros.
We then get good estimates by observing that 
the minimum size of a defining set for $M\in\mathcal{A({\bf s},{\bf t})}$
cannot exceed $\lambda m n$, and
maxsds$({\bf s},{\bf t}) \leq \lambda m n$.  We note that the case
where $\lambda>1/2$ would be easily handled by replacing $\lambda$ with
$1-\lambda$ in the appropriate places, but our statements are simpler
if we do not need to say this each time. For similar reasons, we assume
throughout that $n\le m$.

Since every matrix in $\mathcal{A}({\bf s},{\bf t})$ has a defining set of size $\lambda mn$,
another way to state the conclusion of Theorem \ref{main-result} is
that sds$(M)=\lambda mn-\mathcal{O}(m^{7/4+\eps})$ for
almost all $M\in\mathcal{A}({\bf s},{\bf t})$. It also follows that
$\maxsds({\bf s},{\bf t})=\lambda mn-\mathcal{O}(m^{7/4+\eps})$. These results
are limited to the case when $m,n,{\bf s},{\bf t}$ satisfy the hypotheses of our theorem. Since every $m\times n$ binary matrix has a defining set of size at most $mn/2$, we can also say:

\begin{corollary}\label{cy:maxsdsn}
For $n\le m\le o(n^{1+\eps})$,  we have $\maxsds(n,m)=nm/2 - \mathcal{O}(m^{7/4+\eps})$.
\end{corollary}

Cavenagh \cite{Cav2013} and Cavenagh and Wright \cite{CavWright18}
studied {\em critical sets}, that is, defining sets which are minimal
in the sense that the removal of any element destroys the property of
being a defining set. They showed that the complement of a critical
set is itself a defining set. Therefore Theorem~\ref{main-result}
implies that almost all binary matrices contain no large critical
set. More specifically:

\begin{corollary}\label{cy:critset}
  Under the hypotheses of Theorem~$\ref{main-result}$, almost all
  elements of $\mathcal{A}({\bf s},{\bf t})$ have no critical set of
  size more than $(1-\lambda)mn+\mathcal{O}(m^{7/4+\eps})$.
\end{corollary}

\section{Preliminary results}

In this section we provide some preliminary results used in the proof of Theorem
\ref{main-result}. We utilise the following elegant
characterisation of defining sets from \cite{CavRam18}. It uses the
idea of a South-East walk tracing through a matrix using steps to the
right or downward. Such a walk separates the entries of the matrix into
two classes: those above (and to the right of) the walk and those below (and to the
left of) the walk. In particular, no entry lies on the walk itself.
We say a partial matrix $M \in \mathcal{A}'({\bf s},{\bf t})$ is in
\emph{good form} if whenever $M_{i,j}=1$ and $M_{i,j'}=0$ then $j < j'$
and whenever $M_{i,j}=0$ and $M_{i',j}=1$ then $i < i'$. In other
words, a partial matrix $M \in \mathcal{A}'({\bf s},{\bf t})$ is in
good form if a South-East walk in $M$ exists
with only ones (or empty cells) below the walk and only zeros (or
empty cells) above it.

\begin{theorem}\label{goodform}
Let $M \in \mathcal{A}({\bf s},{\bf t})$ and let $D \in \mathcal{A}'({\bf s},{\bf t})$. Then $D$ is a defining set for $M$ if and only if $D \subseteq M$ and the rows and columns of the partial matrix $M \setminus D$ can be permuted to be in good form.
\end{theorem}

The family of matrices constructed in \cite{CavRam18} have the special property that within any rectangular subarray the difference between the number of ones and zeros is small. This property, combined with Theorem~\ref{goodform}, guarantees no small defining set. In our more general setting, we are interested in the property that the difference between the number of ones and the expected number of ones in any subarray is small. Here, and henceforward, when we refer to the expected number of ones occupying a particular set of cells, the underlying distribution involves a matrix being chosen uniformly at random from all binary matrices with given dimensions and density.

Let $R$ and $C$ be any subsets of the rows and columns, respectively, of $M \in \mathcal{A}({\bf s},{\bf t})$. Let $\lambda$ be the density of $M$. Then the subarray $M[R,C]$ is the $|R| \times |C|$ array of $M$ induced by $R$ and $C$ and $\delta(M[R,C])$ denotes the number of ones in $M[R,C]$ minus $\lambda |R| |C|$, which is the expected number of ones in $M[R,C]$.

\begin{lemma}\label{NickCLemma}
Fix $\eps'>0$ and let $\Delta$ be a function of integers $m$ and
$n$.  Let $M \in \mathcal{A}({\bf s},{\bf t})$ have a density
$\lambda$ satisfying $\eps'\le\lambda\le1/2$. Let $D$ be a
defining set for $M$. If for every $R$ and $C$, subsets of the rows
and columns of $M$, respectively, we have
\begin{equation}\label{e:Delbnd}
  \big|\delta(M[R,C])\big| \leq \Delta(m,n),
\end{equation}
then $|D| \geq \lambda m n -\mathcal{O}(m^{7/4}+m^{1/4}\Delta(m,n)).$
\end{lemma}

\begin{proof}
Let $M \in \mathcal{A}({\bf s},{\bf t})$ be such that $|\delta(M[R,C])| \leq \Delta(m,n)$ for any subsets $R$ and $C$ of the rows and columns, respectively. Let $D \subseteq M$ be a minimal defining set for $M$. We show that the size of $D$ cannot be less than $\lambda m n-\mathcal{O}(m^{7/4}+m^{1/4}\Delta(m,n))$. By Theorem \ref{goodform} we can assume that the rows and columns of $M$ have been permuted so that $M \setminus D$ is in good form.
That is,  we can draw a South-East walk $\mathscr{W}$ in the matrix $M \setminus D$ so that all non-empty cells above $\mathscr{W}$ are zeros and all non-empty cells below $\mathscr{W}$ are ones. Since $D$ is minimal, $M \setminus D$ must contain every one that occurs in $M$ below $\mathscr{W}$ and every zero that occurs in $M$ above $\mathscr{W}$. 

Let $\alpha_0$ and $\alpha_1$ denote the number of zeros and ones (respectively) in $M$ above $\mathscr{W}$, and let $\beta_0$ and $\beta_1$ denote the number of zeros and ones (respectively) in $M$ below $\mathscr{W}$. Hence, we have 
\begin{align*}
\alpha_1+\beta_1 &= \lambda m n, \text{ and}\\
|D| &= \alpha_1+\beta_0.
\end{align*}

We now find an upper bound on
$|\beta_1-\lambda(\beta_1+\beta_0)|$, which is the number of ones
minus the expected number of ones in $M$ below $\mathscr{W}$. For $1\le i\le m$, define $f(i)$ to be the number of cells in row $i$ to the left of
$\mathscr{W}$ and let $f(0)=0$. By definition, the sequence $f(0),\dots,f(m)$ is weakly
increasing. Let $h=\ceil[\big]{m^{3/4}}$. For $1\le i\le \lceil m^{1/4}\rceil$, define a block
$B_i=M[R_i,C_i]$ where 
$R_i=\{ih,\dots,m\}$
and $C_i=\big\{f\big((i-1)h\big)+1,\dots,f(ih)\big\}$. Note that each block $B_i$ lies entirely below 
$\mathscr{W}$ and is disjoint from $B_j$ for $j\ne i$. Moreover,
in any column there are at most $h$ cells that are below $\mathscr{W}$
but are not in any of the $B_i$. For these cells, the difference between the
number of ones and the expected number of ones cannot exceed $nh$, the
total number of cells involved. For each block $B_i$, we then employ the bound
(\ref{e:Delbnd}) to give
$$|\beta_1 - \lambda(\beta_1 + \beta_0)| = nh+\lceil m^{1/4}\rceil\Delta(m,n)
=\mathcal{O}(nm^{3/4}) + \mathcal{O}(m^{1/4}\Delta(m,n)).$$
Now $n \leq m$ and $\lambda \leq 1/2$ with $1/\lambda=\mathcal{O}(1)$, so
\begin{align*}
\beta_ 0 & =  \dfrac{1-\lambda}{\lambda} \beta_1 - \mathcal{O}(m^{7/4}+m^{1/4}\Delta(m,n)) 
\geq \beta_1 - \mathcal{O}(m^{7/4}+m^{1/4}\Delta(m,n)).
\end{align*}
It follows that
\begin{align*}
|D| & =  \beta_0 + \alpha_1 
\geq  \beta_1 + \alpha_1 - \mathcal{O}(m^{7/4}+m^{1/4}\Delta(m,n)) 
= \lambda m n  - \mathcal{O}(m^{7/4}+m^{1/4}\Delta(m,n))
\end{align*}
as claimed.
\end{proof}

Let $\mathcal{N}({\bf s},{\bf t})$ be the number of labelled bipartite graphs with $m$ vertices on one side of the
bipartition with degrees given by ${\bf s}$, and $n$ vertices on the
other side with degrees given by ${\bf t}$. We utilise the following asymptotic estimate from \cite{Canfield}.

\begin{theorem} \label{approx}
Let $m,n,{\bf s},{\bf t},\lambda, A$, and $\eps$ be defined as in Theorem \ref{main-result}. Then we have
$$\mathcal{N}({\bf s},{\bf t}) = {mn \choose \lambda mn}^{-1}  \prod^{m}_{i=1} {n \choose s_i}  \prod^{n}_{j=1} {m \choose t_j}\exp \left( - \mathcal{O}((mn)^{2 \eps}) \right).$$
\end{theorem} 

Lastly, we need the following well-known results called the Chernoff bounds \cite{MitzUp05}. 

\begin{theorem}\label{chernoff}
Let $X_1,\dots,X_n$ be independent Bernoulli random variables where $X_i=1$
with probability $p_i$ and $X_i=0$ with probability $1-p_i$. Let
$X=\sum_{i=1}^n X_i$ and $\mu= \mathbb{E}(X)=\sum_{i=1}^n p_i$. Then
\begin{itemize}
\item[(i)] $\mathbb{P}\big(X \ge (1+\gamma)\mu\big)\le \exp(-\frac{\gamma^2}{2+\gamma}\mu)$ for all $\gamma>0$,
\item[(ii)] $\mathbb{P}\big(|X-\mu| \ge \gamma \mu\big) \le 2\exp(-\mu \gamma^2/3)$ 
for all $0<\gamma<1$.
\end{itemize}
\end{theorem}

\section{Proof of the main result}\label{s:prfmain}

An element of $\mathcal{A}({\bf s},{\bf t})$ is the bi-adjacency
matrix of a bipartite graph with $m$ vertices on one side of the
bipartition with degrees given by ${\bf s}$, and $n$ vertices on the
other side with degrees given by ${\bf t}$. We define the density of a bipartite graph to be the density of its
bi-adjacency matrix.

Let $A$ and $B$ be subsets of the vertices of $G$, each from a different side and denote the number of edges between $A$ and $B$ by $e(A,B)$. The property (\ref{e:Delbnd}) is equivalent to the difference between the number of edges and the expected number of edges between $A$ and $B$ being at most $\Delta(m,n)$. Therefore, by Lemma \ref{NickCLemma}, the following theorem implies our main result, Theorem~\ref{main-result}.

\begin{theorem}\label{main}
Let $G({\bf s},{\bf t})$ be chosen uniformly at random from the bipartite graphs with one side of the bipartition of size $m$ with degrees from ${\bf s}$ and the other side of size $n$ with degrees from ${\bf t}$. Let $\lambda$ be the density of $G({\bf s},{\bf t})$ and suppose that $m$, $n$, ${\bf s}$, ${\bf t}$ and $\lambda$ satisfy the hypotheses of Theorem~$\ref{main-result}$. Then there is some constant $c>0$ such that,
with probability $1-o(1)$,
$$\big| e(A,B) - \lambda |A||B|\big |  \leq c(mn^{1/2+\eps}+nm^{1/2+\eps}),$$
for any two subsets $A$ and $B$ of the vertices, each from a different side.
\end{theorem}

Let $G(n,m,\lambda)$ be a random bipartite graph with sides of size $m$ and $n$, in which each of the $mn$ possible edges occurs independently with probability $\lambda$. Note that with respect to this graph, the expectation of $e(A,B)$ is $\lambda |A||B|$, for any two subsets $A$ and $B$ of the vertices of $G(n,m,\lambda)$, each from a different side. 

\begin{proof}[Proof of Theorem \ref{main}]

Fix a positive constant $c$.
We say a bipartite graph \emph{has property} $\mathcal{P}$, if there exist subsets $A$ and $B$, each from a different side, such that
$$\big | e(A,B) - \lambda |A||B|\big |  > c(mn^{1/2+\eps}+nm^{1/2+\eps}).$$
Let $\mathbb{P}_{{\bf s},{\bf t}}(\mathcal{P})$ denote the probability that $G({\bf s},{\bf t})$ has property $\mathcal{P}$ and let $\mathbb{P}_{\lambda}(\mathcal{P})$ be the probability that $G(n,m,\lambda)$ has property $\mathcal{P}$.
We define $E_{\bf s,t}$ to be the event that $G(m,n,\lambda)$ has degree sequence $({\bf s},{\bf t})$.
Then
$$
\mathbb{P}_{{\bf s},{\bf t}}(\mathcal{P})=\mathbb{P}_{\lambda}(\mathcal{P}\,|\,E_{{\bf s},{\bf t}}) 
\leq \frac{\mathbb{P}_{\lambda}(\mathcal{P})}{\mathbb{P}_{\lambda}(E_{{\bf s},{\bf t}})}.  
$$

We claim that $\mathbb{P}_{{\bf s}, {\bf t}}(\mathcal{P})$ goes to zero as $m,n \rightarrow \infty$. Firstly we find a lower bound on $\mathbb{P}_{\lambda}(E_{{\bf s},{\bf t}})$.  To simplify our calculations we let $\lambda' = (1-\lambda)$. Applying Stirling's formula to the binomials given in Theorem \ref{approx}, we have the following approximations, provided $m$, $n$, ${\bf s}$ and ${\bf t}$ satisfy the hypotheses of Theorem~\ref{main-result}:
\begin{align*}
{ mn  \choose \lambda mn} 
& =  \exp \Big( -(\lambda \log \lambda + \lambda' \log \lambda')mn -\mathcal{O}(\log(mn)) \Big),
\end{align*}
\begin{align*}
  \prod_{i=1}^m {n \choose s_i} & = \exp \bigg( mn\log n -\sum_{i=1}^m s_i \log s_i - \sum_{i=1}^m (n-s_i)\log(n-s_i) -\mathcal{O}(m\log n)\bigg).
\end{align*}
By assumption, for $1 \leq i \leq m$ we have $s_i=\lambda n + s_i'$ where $s_i' = \mathcal{O}(n^{1/2+\eps})$ and $\sum_{i=1}^m s_i' = 0$. Hence we have
\begin{align*}
\sum_{i=1}^m s_i \log s_i 
& = \lambda mn \log (\lambda n) + \sum_{i=1}^m (\lambda n + s_i ')\log \left(1+\dfrac{s_i'}{\lambda n}\right) 
= \lambda mn \log (\lambda n) + \mathcal{O}(mn^{2\eps}).
\end{align*}
Similarly,
\begin{align*}
 \sum_{i=1}^m (n-s_i) \log (n-s_i) 
= \lambda' mn \log (\lambda' n) + \mathcal{O}(mn^{2\eps}).
\end{align*}
Hence, we have
\begin{align*}
  \prod_{i=1}^m {n \choose s_i} 
 & = \exp \left(-(\lambda \log \lambda + \lambda' \log \lambda')mn- 
\mathcal{O}(mn^{2\eps})\right).
\end{align*} 
A similar argument yields
\begin{align*}
  \prod_{j=1}^n {m \choose t_j}  
& = \exp \left(-(\lambda \log \lambda + \lambda' \log \lambda')mn- 
\mathcal{O}(nm^{2\eps}) \right).
\end{align*}
Combining all of the above approximations with Theorem \ref{approx}, we find that
$$\mathcal{N}({\bf s},{\bf t}) = \exp \left( 
-(\lambda \log \lambda + \lambda' \log \lambda')mn
-\mathcal{O}(mn^{2\eps}+nm^{2\eps}) \right). $$
There are $\binom{mn}{\lambda mn}$ labelled bipartite graphs with sides of size $m$ and $n$ and density $\lambda$, so
\begin{align} \label{lb}
\mathbb{P}_{\lambda}(E_{{\bf s},{\bf t}}) 
=\dfrac{\mathcal{N}({\bf s},{\bf t})}{\binom{mn}{\lambda mn}}
=\exp \left( 
-\mathcal{O}(mn^{2\eps}+nm^{2\eps}) \right).
\end{align}

We now need to find an upper bound on  $\mathbb{P}_{\lambda}(\mathcal{P})$. Let $N=\floor[\big]{(c/\lambda)(mn^{1/2+\eps}+nm^{1/2+\eps})\,}$. Then we have
\begin{align}
\mathbb{P}_{\lambda}(\mathcal{P})
&\leq\mathbb{P}\left(\exists \,A,B\text{ such that }\big | e(A,B) - \lambda|A||B|\big |  > \lambda N\right) \nonumber\\ 
&\leq \sum\limits_{A,B}\mathbb{P}\left(\big| e(A,B) - \lambda |A||B|\big |  > \lambda N\right)  \nonumber\\ 
& = \sum\limits_{|A||B| >  N}\!\!\!\mathbb{P}\left( \big | e(A,B) - \lambda |A||B|\big |  > \lambda N \right) 
+\sum\limits_{|A||B| \le N}\!\!\!\mathbb{P}\left(\big | e(A,B) - \lambda |A||B| \big |  > \lambda N \right)\!,
\label{e:f1}
\end{align}
where the first inequality follows from the union bound and this sum is over all pairs of subsets $A$ and $B$ of the vertices of $G(n,m,\lambda)$, each from a different side. 
By Theorem \ref{chernoff} {\it (ii)},
\begin{align}
\sum\limits_{|A||B| > N}\mathbb{P}\left(\big | e(A,B) - \lambda|A||B|\big |  > \lambda N\right)
& = \sum_{k=N+1}^{mn} \,\, \sum_{|A||B|=k} \mathbb{P}\left(\big | e(A,B) - \lambda |A||B| \big |  > \lambda N \right)\nonumber\\ 
& \le \sum_{k=N+1}^{mn} \,\, \sum_{|A||B|=k} 2 \exp \left( -\lambda N^2/(3k) \right)\nonumber\\ 
& \le \exp \left( -\lambda N^2/(3mn) \right) \sum_{k=N+1}^{mn} \sum_{|A||B|=k} 2 \nonumber\\ 
& \le \exp \left( -\lambda N^2/(3mn) \right) 2^{m+n+1},\label{e:f2}
\end{align}
where the last inequality is due to the fact that the number of pairs $A,B$ is bounded above by $2^{m+n}$. We now move on to those subsets satisfying $|A||B| \le N$.  Note that in this case, if $e(A,B) - \lambda|A||B|\leq0$ then $\big| e(A,B) - \lambda|A||B|\big|\leq\lambda|A||B| \leq \lambda N$. Thus, by Theorem \ref{chernoff}{\it (i)}, we have
\begin{align}
\sum\limits_{|A||B| \le N}\mathbb{P}\left(\big | e(A,B) - \lambda|A||B|\big |  > \lambda N\right)
& = \sum_{k=1}^N \,\, \sum_{|A||B|=k} \mathbb{P}\left(\big | e(A,B) - \lambda |A||B| \big |  > \lambda N \right)
\nonumber\\ 
& \le \sum_{k=1}^N \,\, \sum_{|A||B|=k} \exp \left( -\frac{\lambda N^2}{2k+N}\right)
\nonumber\\ 
& \le \exp \left(-\lambda N^2/(3N)\right) \sum_{k=1}^N \,\, \sum_{|A||B|=k} 1
\nonumber\\ 
& \le \exp \left( -\lambda N/3 \right)2^{m+n}.\label{e:f3}
\end{align}
Combining (\ref{e:f1}), (\ref{e:f2}) and (\ref{e:f3}), we have
\begin{align*}
\mathbb{P}_{\lambda}(\mathcal{P})
\leq \exp \left(-\frac{\lambda N^2}{3mn}(1+o(1))\right).
\end{align*}
We can choose $c$ large enough that
$\lambda N^2/(3mn)$ exceeds any fixed multiple of $(mn^{2\eps} +nm^{2\eps})$.
In comparison, (\ref{lb}) is independent of $c$. So for an appropriately large $c$,
\begin{align*}
\frac{\mathbb{P}_{\lambda}(\mathcal{P})}{\mathbb{P}_{\lambda}(E_{{\bf s},{\bf t}})} & \leq \exp\left(\mathcal{O}(mn^{2\eps} +nm^{2\eps})- \lambda N^2/(3mn)\right)=o(1). 
\end{align*}
Hence $\mathbb{P}_{{\bf s},{\bf t}}(\mathcal{P})$ tends to zero as $m,n \rightarrow \infty$ and we are done.
\end{proof}

A topic for future research might be to try to identify which matrices
have the largest sds and what structure those defining sets have. Our
proofs do not give much insight into these questions. However, we do
at least know that $\lambda$ must be very close to $1/2$ in order to
achieve $\maxsds(m,n)$.

  \let\oldthebibliography=\thebibliography
  \let\endoldthebibliography=\endthebibliography
  \renewenvironment{thebibliography}[1]{
    \begin{oldthebibliography}{#1}
      \setlength{\parskip}{0.4ex plus 0.1ex minus 0.1ex}
      \setlength{\itemsep}{0.4ex plus 0.1ex minus 0.1ex}
  }
  {
    \end{oldthebibliography}
  }

\end{document}